\documentclass{amsart}
\pdfoutput=1

\usepackage{tikz, hyperref, microtype, xcolor, amssymb}
\usepackage[sc]{mathpazo}
\hypersetup{
    colorlinks,
    linkcolor={red!80!black},
    citecolor={blue!80!black},
    urlcolor={blue!80!black}
}

\newcommand{\N}{\mathbb{N}}
\newcommand{\mP}{\mathcal{P}}
\newcommand{\mQ}{\mathcal{Q}}

\newtheorem{theorem}{Theorem}
\newtheorem{corollary}{Corollary}
\newtheorem{lemma}{Lemma}

\theoremstyle{definition}
\newtheorem{example}{Example}
\newtheorem{definition}{Definition}

\title{Simultaneous cores with restrictions and a question of Zaleski and Zeilberger}
\author{Paul Johnson}
\address{University of Sheffield}
\email{paul.johnson@sheffield.ac.uk}
\begin{document}
\maketitle
\begin{abstract}
  The main new result of this paper is to count the number of $(n,n+1)$-core partitions with odd parts, answering a question of Zaleski and Zeilberger \cite{ZZodd} with bounty a charitable contribution to the OEIS.  Along the way, we prove a general theorem giving a recurrence for $(n, n+1)$-core parts whose smallest part $\lambda_\ell$ and consecutive part differences $\lambda_i-\lambda_{i+1}$ lie in an arbitrary $M\subset \N$ which unifies many known results about $(n, n+1)$-core partitions with restrictions.  We end with discussions of enrichments of this theorem that keep track of the largest part, number of parts, and size of the partition, and about a few cases where the same methods work on more general simultaneous cores.
  
\end{abstract}

The paper is laid out as follows.  Section \ref{sec:intro} gives a quick background to the question and states our main theorems.  Section \ref{sec:recursion} reviews the abacus construction, and uses it to prove Theorem \ref{thm: }, a general theorem giving recrusions for simultaneous $(n,n+1)$-cores meeting certain restrictions.  Section \ref{sec:odd} relates the count of simultaneous cores with odd parts to that of even parts, answering a question of Zeilberger and Zaleski. Section \ref{sec:enriched} briefly discusses enriching the recursion of the main theorem to take into account other partitions, and Section \ref{sec:ab} briefly touches on a few cases where the basic recursion can move beyond $(n, n+1)$-cores.  

\section{Introduction} \label{sec:intro}
A \emph{partition of $n$} is a nonincreasing sequence $\lambda_1\geq \lambda_2\geq\cdots\geq \lambda_{k}> 0$ of positive integers such that $\sum \lambda_i=n$. We call $n$ the $\emph{size}$ of $\lambda$ and denote it by $|\lambda|$; we call
$k$ the \emph{length} of $\lambda$ and denote it by $\ell(\lambda)$.

We frequently identify $\lambda$ with its Young diagram; there are many conventions for this.  We draw $\lambda$ in the first quadrant, with the parts of $\lambda$ as the columns of a collection of boxes.
\begin{definition}
The \emph{arm} $a(\square)$ of a cell $\square$ is the number of cells contained in $\lambda$ and above $\square$, and the \emph{leg} $l(\square)$ of a cell is the number of cells contained in $\lambda$ and to the right of $\square$.
The \emph{hook length} $h(\square)$ of a cell is $a(\square)+l(\square)+1$.
\end{definition}
\begin{example}
The cell $(2,1)$ of $\lambda=3+2+2+1$ is marked $s$; the cells in the
leg and arm of $s$ are labeled $a$ and $l$, respectively.
\begin{center}
\begin{tikzpicture}[scale=.5]
\draw[thin, gray] (0,0) grid (1,3);
\draw[thin, gray] (1,0) grid (3,2);
\draw[thick] (0,0)--(0,3)--(1,3)--(1,2)--(3,2)--(3,1)--(4,1)--(4,0)--cycle;
\draw (1.5,.5) node{$s$};
\draw (1.5,1.5) node{$a$};
\draw (2.5,.5) node{$l$};
\draw (3.5,.5) node {$l$};
\draw (8.5,1.5) node[align=left] {$a(s)=\# a=1$ \\ $l(s)=\# l=2$ \\ $h(s)=4$};
\end{tikzpicture}
\end{center}
\end{example}
\begin{definition}
An $a$-\emph{core} is a partition that has no hook lengths of size $a$. An $(a,b)$-\emph{core} is a partition that is simultaneously an $a$-core and a $b$-core.  We use $\mathcal{C}_{a,b}$ to denote the set of $(a,b)$-cores.
\end{definition}
\begin{example}
We have labeled each cell $\square$ of $\lambda=3+2+2+1$ with its hook length $h(\square)$.
\begin{center}
\begin{tikzpicture}[scale=.5]
\draw[thin, gray] (0,0) grid (1,3);
\draw[thin, gray] (1,0) grid (3,2);
\draw[thick] (0,0)--(0,3)--(1,3)--(1,2)--(3,2)--(3,1)--(4,1)--(4,0)--cycle;
\draw (.5,.5) node{$6$};
\draw (1.5,.5) node{$4$};
\draw (2.5,.5) node{$3$};
\draw (3.5,.5) node{$1$};
\draw (.5,1.5) node {$4$};
\draw (1.5,1.5) node{$2$};
\draw (2.5,1.5) node{$1$};
\draw (.5,2.5) node{$1$};
\end{tikzpicture}
\end{center}
We see that $\lambda$ is \emph{not} an $a$-core for $a\in \{1,2,3,4,6\}$;
but \emph{is} an $a$-core for all other $a$.
\end{example}

The study of simultaneous core partitions began with the work of Anderson. \cite{anderson}, who proved that when $a$ and $b$ are relatively prime, there $\frac{1}{a+b}\binom{a+b}{a}$ of them.  Her proof uses the abacus to get a bijection with rational Dyck paths.   The study of simultaneous core partitions has exploded in recent years, largely due to the work of Amstrong, Hanusa and Jones \cite{AHJ} relating them to rational Catalan combinatorics.

One direction of this study initiated by Amdeberhan \cite{Amdeberhan}, is to study simultaneous core partitions that are simultaneously core for more than two integers, and simultaneous core partitions that satisfy more classical conditions on partitions.  In \cite{Amdeberhan} made many conjectures about counts of such simultaneous core partitions, which now seem to all be proven.  Our first result is a general theorem that gives a unified recursion that proves many of Amdeberhan's conjectures.  

The two most important examples for us are the following.  First, Amdeberhan conjectured that the number of $(n, n+1)$-core partitions with distinct parts was the $n+1$st Fibonacci number $F_{n+1}$; this was proven independently by Xiong \cite{XiongDistinct} and Straub \cite{Straub}.

Second, Amdeberhan conjectured a formula for the number $P_d(n)$ of $(n, n+1)$-core partition all of whose parts are divisible by $d$, which was proven by Zhou \cite{ZhouRaney}.  Specifically, if $n=qd+r$ with $0\leq r<d$, then 
$$P_d(n)=\frac{r+1}{n+1}\binom{n+q}{n}$$
which are known as Raney or Fuss-Catalan numbers, $P_d(n)=A_q(d-1, r-1)$ in the notation in wikipedia. 

Our perspective is to unify both of these conditions (having distinct parts, and having all parts divisible by $d$) into the general family of conditions given by restricting the difference between consecutive parts of the partitions.  Let $\mathbb{N}=\{0,1,2,\dots\}$ be the natural numbers, and let $\mathcal{P}$ denote the set of all partitions. Let $\ell(\lambda)$ be the number of parts, and define $\lambda_{\ell+1}=0$ by convention.  We denote the empty partition by $0$.

\begin{definition}
Let $M\subset\mathbb{N}$.  We define
  $$\mathcal{P}^M=\{\lambda\in\mP :  \lambda_i-\lambda_{i+1}\in M \text{ for } 1\leq i \leq\ell\-1\}$$
  $$\mathcal{Q}^M=\{\lambda\in\mP :  \lambda_i-\lambda_{i+1}\in M \text{ for } 1\leq i \leq \ell\}$$
That is, $\mP^M$ is the set of partitions where the difference between consecutive parts is in $M$, and $\mQ^M\subset\mP^M$ asks that in addition that the smallest part $\lambda_\ell$ is in $M$.
\end{definition}

From our perspective, $\mQ^M$ is more natural, and will be focused on that. For most choices of $M$, $\mP^M$ and $\mQ^M$ are very unnatural sets of partitions; we highlight a few choices of $M$ that give more familiar classes of partitions.

\begin{example} Let $\N_+=\{1,2,\dots\}$.  Then $\mP^{\N_+}=\mQ^{\N_+}$ is the set of partitions with distinct parts.
\end{example}

\begin{example}
For $k>1$ an integer, let $k\N=\{0,k,2k,\dots \}$.  Then $\mathcal{P}^{k\N}$ is the set of partitions with all parts congruent modulo $k$, and $\mQ^{k\N}$ is the set of partitions with all parts divisible by $k$.
  In particular, $\mP^{2\N}\setminus \mQ^{2\N}\cup {0}$ is the set of partitions with odd parts.
\end{example}
\begin{example}
Combining the first two, taking $M=k\N_+=\{k,2k,3k,\dots\}$, then $\mQ^{k\N_+}$ consists of partitions with distinct parts, all of which are divisible by $k$, and $\mP^{k\N_+}$ consists of partitions with distinct parts, all of which are congruent mod $k$.
\end{example}

\begin{example} Let $\N+r=\{r,r+1,r+2,\dots\}$ denote the set of integers great than or equal to $r$.  Then $\mP^{\N+r}$ is the set of partitions with difference between consecutive parts at least $r$, and $\mQ^{\N+r}$ asks, in addition, that the smallest part is at least $r$.  When $r=2$, these two classes of partitions take part in the Rogers-Ramanujan identities:
\begin{align*}
  \sum_{\lambda\in \mP^{\N+2}} q^{|\lambda|}&=\prod_{m\geq 0} \frac{1}{(1-q^{5m+1})(1-q^{5m+4})} \\
  \sum_{\lambda\in \mQ^{\N+2}} q^{|\lambda|}&=\prod_{m\geq 0} \frac{1}{(1-q^{5m+2})(1-q^{5m+3})}
  \end{align*}
  
\end{example}

\begin{example} \label{ex:rdistinct}
  Let $[r]=\{0,1,\dots, r\}$.  Then $\lambda\in\mQ^{[r]}$ if and only if $\lambda^T$, the conjugate of $\lambda$, has no parts repeating more than $r$ times.  These partitions take part in Glaisher's theorem:
  $$\sum_{\lambda\in\mQ^{[r]}}=\prod_{r+1\nmid m} \frac{1}{1-q^m}$$
in case $r=1$, this is Euler's theorem that the number of partitions of $n$ into distinct parts is equal to the number of partitions of $n$ into odd parts.
\end{example}

Our general theorem is a recurrence for $|\mathcal{C}_{n,n+1}\cap \mQ^M|$ for \emph{any} subset $M$.  The recurrence has a different form depending on whether or not $0\in M$.  

\begin{theorem}
  For $M\subset \N$, let $C^M(n)=|\mathcal{C}_n\cap \mQ^M|$ be the count of simultaneous $(n, n+1)$-core partitions with smallest part and consecutive part differences in $M$.  Then $C^M(0)=1$ (the empty partition), and $C^m(n)$ is determined from this by the recurrence:
  \begin{enumerate}
  \item If $0\in M$,
    $$C^M_{n+1}=\sum_{\substack{k\in M\\k\leq n}}C^M_kC^M_{n-k}$$
  \item If $0\notin M$,
    $$C^M_{n+1}=1+\sum_{\substack{k\in M\\k\leq n}}C^M_{n-k}$$
  \end{enumerate}
\end{theorem}

We might term these recurrences the restricted recurrence Catalan and Fibonacci sequences, respectively; when $M=\N$, the recurrence is the standard recurrence for the Catalan numbers; when $M=\N\setminus 0$, the recurrence is equivalent to the familiar formula for the sum of the first $n$ Fibonacci numbers:
$$F_{n+1}=1+\sum_{k=1}^{n-1} F_{n-k}$$
When $M$ is only a subset of these two sets, the recurrence only has those terms corresponding to elements of $M$.

Both identities can be rephrased in terms of generating series, though if $0\in M$ the statement requires the Hadamard product: recall if $f=\sum_{i\geq 0} a_nq^n$ and $g=\sum_{m\geq 0} b_jq^m$ are formal power series, the Hadamard product $(f\star g)=\sum_{n\geq 0} a_nb_bq^n$.
\begin{corollary} For $M \subset \N$, let $$\chi_M=\sum_{m\in M} q^m$$
  $$C^M(q)=\sum_{n\geq 0} C^M_nq^n$$ 

  Then if $0\in M$, we have
    $$C^M(q)=1+q(C_M(q)\star \chi_M(q))C_M(q) $$

  If $0\notin M$, then 
    $$C_M(q)=\frac{1}{1-q}+q\chi_m(q)C_m(q)$$
    and hence
    $$C_m(q)=\frac{1}{1-q}\frac{1}{1-q\chi_m(q)C_m(q)}$$
\end{corollary}

Before we give the proof, we illustrate some examples of this theorem.

\begin{corollary}Let $D(n)$ denote the number of $(n, n+1)$-cores with distinct parts, then $D(n)=|\mathcal{C}_n\cap \mQ^{\N\setminus 0}|$, and so we have
  $$D(n+1)=1+\sum_{k=1}^n D(n-k)$$
  In particular, for $n=0$ the sum on the right hand side is empty, and so $D(1)=1$, and otherwise, removing the first term of the sum and reindexing gives
  $$D(n+1)=1+\sum_{k=1}^n D(n-k)=D(n-1)+[1+\sum_{k=1}^{n-1} D(n-1-k)]=D(n-1)+D(n)$$
  and so $D(n)=F(n)$, the Fibanocci numbers.

\end{corollary}

\begin{corollary}Let $D_r(n)$ denotes the number of $(n,n+1)$-core with no part occuring more than $r$-times; by Example \ref{ex:rdistinct} we have $D_r(n)=|\mathcal{C}_n\cap \mQ^{[r]}|$, and hence  
$$D_{r}(n+1)=\sum_{k=0}^r D_r(k)D_r(n-k)$$
For $n\leq r$, this is the usual recurrence for Catalan numbers, and so in this range we have $D_r(n)=C_n$.  
If we take $r=1$, this directly recovers the standard Fibonacci recurrence for simultaneous cores with distinct parts, since $D_1(0)=D_1(1)=1$.  In general, since the recursion does not differ from the standard Catalan recursion until $n>r$, we have $D_r(n)=C_n$ for $n\leq r$, and hence in general $D_r(n+1)=\sum_{k=0}^r C_k D_r(n-k)$.  Taking $r=2$, we have
$$D_2(n+1)=D_2(n)+D_2(n-1)+2D_2(n-2)$$

\end{corollary}

\begin{corollary}Let $P_d(n)$ denotes the number of $(n, n+1)$-core partitions with all parts divisible by $d$; this is the case $M=\{0,d,2d,3d,\dots\}$.  The recurrence becomes:
$$P_d(n+1)=\sum_{k=0}^{\lfloor n/d\rfloor} P_d(dk)P_d(n-dk)$$
  which is Riordan's recurrence for the Fuss-Catalan numbers. In particular, in this case the Theorem reproves Amdeberhan's conjecture that, if $n=qd+r$ with $0\leq r<d$, then
  $$P_d(n)=\frac{r+1}{n+1}\binom{n+q}{q}$$
\end{corollary}

\begin{corollary} Let $D_d(n)$ denote the number of $(n, n+1)$-core partitions with distinct parts, all divisible by $d$; this is the case $M=\{d,2d,3d,\dots\}$. Let $D_d(q)$ be the corresponding generating function.  We have $\chi_M(q)=q^d/(1-q^d)$, and a little algebra gives
  $$D_d(q)=\frac{1+q+q^2+\cdots +q^{d-1}}{1-q^d-q^{d+1}}$$
 and hence $D_d(n)$ is the series determined by the recurrence $D_d(n)=D_d(n-d)+D_d(n-d-1)$ and the initial values $D_d(0)=D_d(1)=\cdots=D_d(d-1)=1$; when $d=2$ this is $1,1,1,2,2,3,4,5,7,9,12,16,\dots$ which up to reindexing is \href{https://oeis.org/A000931}{A000931} the Padovan numbers.

\end{corollary}

Finally, Let $O(n)$ and $E(n)$ be the number of $(n,n+1)$ cores with all parts odd or even, respectively.  Then $E(n)=P_2(n)$ is already determined by the previous recurrence, and analysis of the abacus allows us to derive the formula
$$E(n+2)=2O(n)-O(n-2)$$
and thus determine $O(n)$; our derivation of this last expression is, at present, not as elegant as the expression seems to merit.

\section{The abacus construction}
In this section we briefly recall the abacus construction for $n$-core partitions.

Let $w(\lambda)=a_1a_2a_3a_4\cdots$ be the boundary path of $\lambda$, with each $a_i$ being either $E$ or $S$, normalized so $a_1=E$ is the first $E$ step, and padded out with trailing $E$s if necessary.  Then $\lambda$ is an $n$ core if and only if $a_k=E$ implies that $a_{k+n}=E$.  

The general idea behind the $n$-abacus construction is to split the $a_i$ among $n$ rows periodically.  If we write $i=kn+r$ with $0\leq k, 1\leq r\leq n$, then we will put a black circle at the $r$th row and $k$th column if $a_i=E$, and a white circle in that spot if $a_i=S$.  Note that this means our rows start with 1 but our columns start with 0.

On the $n$-abacus, the condition for being an $n$-core translate to once a spot on a given row is black, every later spot on that row must also be black.  Our choice of normalization means that for $n$-cores, the first row will always be filled entirely with black dots.    

Thus, we can record the $n$-core as a function $f:\{0,1,\dots, n-1\}\to \mathbb{N}$ by letting $f(i)$ be the first column where the $i$th row has a black dot.  Note that we always have $f(0)=0$, giving a bijection from $n$-cores to $\N^{n-1}$, but including this value will prove useful for describing $(n, n+1)$-cores.

The condition to see an $n+1$-core on the $n$ abacus is that if the circle on the $i$th row and $j$th column is black, with $i<n$, then the circle at the $i+1$st row and $j+1$st column is also black.  In terms of the function $f$, this condition is $f(i+1)\leq f(i)+1$.    In addition to this, if the bead on the $n$th row (last row) and $j$th column is filled, then so must be the bead on the 1st row and $j+2$ column.  Since we only work with $n$-cores, the first row is always completely filled, and this condition is vacuous.  This is a large part of what makes $(n, n+1)$-cores easier to handle than general $(a,b)$-cores.

Summarizing, we have shown the following:
\begin{lemma}The set of $(n, n+1)$-core partitions are in bijection with functions $f:\{0,\dots, n-1\}\to \mathbb{N}$ satisfying:
\begin{enumerate}
\item $f(0)=0$
\item $f(k+1)\leq f(k)+1$
\end{enumerate}
\end{lemma}

\section{Proof of the general theorem}

\subsection{Main ideas} Proving theorem \ref{ } depends on two things: understanding how the standard Catalan number recurrence is proven in terms of the abaci functions, and understanding how to spot whether or not an $n$-core is in $\mQ^M$ from its abaci.

The standard Catalan recurrence is often proven in terms of decomposing Dyck paths by the first return to the diagonal.  There is a natural bijection from abaci functions to Dyck paths -- $f(k)$ measures how far beneath the line $y=x$ the Dyck path is at the $k$th step.  Thus, we can prove the standard recurrence for Catalan numbers by decomposing the abaci functions by their first nontrivial 0.

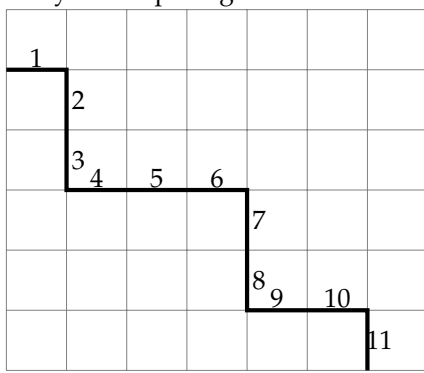
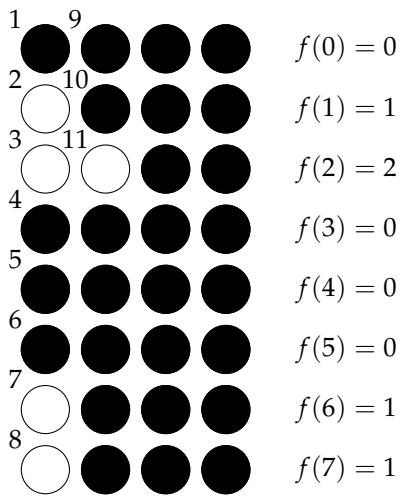
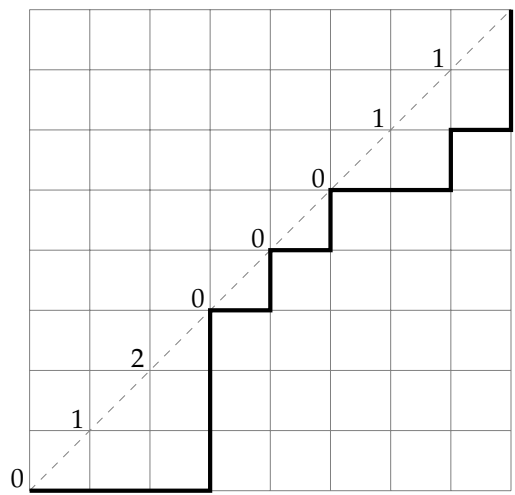
\begin{tikzpicture}[scale=.8]
\foreach \x in {0,1,...,3}           
\foreach \y in {0,1,...,7}
\draw (\x,\y) circle (.4);
\begin{scope}
  \clip (-.5,7.5)--(-.5,6.5)--(.5,6.5)--(.5,5.5)--(1.5,5.5)--(1.5,4.5)--(-.5,4.5)--(-.5,1.5)--(.5,1.5)--(.5,-.5)--(3.5,-.5)--(3.5,7.5)--(-.5,7.5);
\foreach \x in {0,1,...,3}           
\foreach \y in {0,1,...,7}
\filldraw (\x,\y) circle (.4);
\end{scope} 
\begin{scope}[shift={(-.5,8.5)}, yscale=-1]
\foreach \i in {1,...,8}
 \node at (0,\i) {\i};
\end{scope}
\node at (.5,7.5) {9};
\node at (.5,6.5) {10};
\node at (.5,5.5) {11};
\begin{scope}[xshift=5cm]
  \node at (0,7) {$f(0)=0$};
  \node at (0,6) {$f(1)=1$};
  \node at (0,5) {$f(2)=2$};
  \node at (0,4) {$f(3)=0$};
  \node at (0,3) {$f(4)=0$};
  \node at (0,2) {$f(5)=0$};
  \node at (0,1) {$f(6)=1$};
  \node at (0,0) {$f(7)=1$};
  \end{scope}

\begin{scope}[xshift=7cm]
  \draw[very thin, gray] (0,0) grid (8,8);
  \draw[very thin, dashed, gray] (0,0)--(8,8);
  \draw[ultra thick](0,0)--(3,0)--(3,3)--(4,3)--(4,4)--(5,4)--(5,5)--(7,5)--(7,6)--(8,6)--(8,8);
  \node at (-.2,.2) {0};
  \node at (.8,1.2) {1};
  \node at (1.8,2.2) {2};
  \node at (2.8,3.2) {0};
  \node at (3.8,4.2) {0};
  \node at (4.8,5.2) {0};
  \node at (5.8,6.2) {1};
    \node at (6.8,7.2) {1};

\end{scope}

\begin{scope}[yshift=10cm]
\draw[very thin, gray]  (0,0)  grid (7,6);
\draw[ultra thick] (0,5)--(1,5)--(1,3)--(4,3)--(4,1)--(6,1)--(6,0);
\node at (.5,5.2) {1};
\node at (1.2, 4.5) {2};
\node at (1.2, 3.5) {3};
\node at (1.5,3.2) {4};
\node at (2.5,3.2) {5};
\node at (3.5,3.2) {6};
\node at (4.2, 2.5) {7};
\node at (4.2, 1.5) {8};
\node at (4.5,1.2) {9};
\node at (5.5,1.2) {10};
\node at (6.2, .5) {11};

\end{scope}

\end{tikzpicture}

To recognize whether or not the partition corresponding to a given abacus is in $\mQ^M$, we must discuss how to recognize consecutive part differences $\lambda_i-\lambda_{i+1}$ in terms of the abacus.  

\begin{lemma} \label{lem:rubbeads} An $n$-core partitions $\lambda$ is in $\mQ^M$ if and only if the length of every run of white beads in a column is in $M$.
  \end{lemma}
\begin{proof}
Since the black circles are horizontal steps, and the white circles are vertical steps, the differences between consecutive parts (and the smallest part) are exactly the lengths of the runs of consecutive white beads, that is, the number of white beads between any two consecutive black beads on the abacus.  

Since the first row is all black beads, none of these runs of beads will run across more than one column. 
\end{proof}

The proof of Theorem 1 is simply combining the standard derivation of the Catalan recurrence with Lemma \ref{lem:runbeads}
\subsection{Proof of Theorem 1}
Let $f$ be the abacus function of an $(n+1,n+2)$-core partition.  

The proof is essentially in the diagram above, and the picture as follows:

First, we consider the case $0\notin M$.  First, we argue that in the second column, we already have every bead is black.  Indeed, the top bead of the second column is black because $\lambda$ is an $n$-core, and the second bead of the second column is black because $\lambda$ is an $n+1$-core, but since $0\notin M$ if we ever have two consecutive black beads then all the beads from that point on must be black.  So $\lambda$ is determined by its first column.

Now, consider the first column, although $0\notin M$, there is still one partition where the second bead could also be black -- the empty partition.  Otherwise, the second bead must occur after $k$ white beads, where $k\in M$.  Removing the first $k+1$ rows, the remaining $n+1-(k+1)=n-k$ rows must be the abacus of an $(n-k, n-k+1)$ core, giving the recurrence. 

This paragraph still holds in the case $0\in M$; now deleting the first row and first column from the first $k+1$ rows gives the abacus of a $k, k+1$ core.

Now, consider the first column; the top bead is black.  

In particular, looking after the first column; the dot on runner 0 is black, and so there must be another black dot on at least one of runner $1,2...r+1$, or we would have $r+1$ consecutive white dots in a column.

Let $\ell>0$ be the first positive runner that has a black dot.  Then, on the one hand we can read runners $\ell$ to $r-1$ as an $n-\ell$ abacus.  On the other hand, looking at runners $0..\ell-1$, the first column is fixed as a black dot and $\ell-1$ white dots.  Looking at the second column, the first two dots must be black.  So, looking at the first

\section{Simultaneous Cores with all odd parts} We now turn to addressing simultaneous cores with all odd parts.  

\begin{theorem}Let $E(n) (O(n))$ be the number of simultaneous $(n,n+1)$-cores with all parts even (odd, respectively).  Then
  $$2O(n)-O(n-2)=E(n+2)$$
  Since we can calculate $E(n)$, this determines $O(n)$.
\end{theorem}

Our proof of this theorem seems needlessly complicated -- presumably there's a more direct proof.

\begin{proof}
  If $\lambda$ has all odd parts, then $\lambda_{i}-\lambda_{i+1}$ is always even, except for $\lambda_{\ell}-\lambda_{\ell+1}=\lambda_\ell$, the smallest part.



  In terms of the abacus diagram, any run of white dots has even length, except the last run has odd length.  Thus, we see that $n$-core partitions with all parts odd seem very closely related to $n$-core partitions with all parts even, and one might hope to relate them by just adding or removing an extra white bead to the last run.  If this run is in the first column, this works fine.  However, if the last run is not in the first column, then adding an extra column here will make one of the earlier columns have an odd run of white beads.

  \begin{lemma} Let $DE(n)$ and $DO(n)$ be the number of $(n,n+1)$-core partitions with all parts even and odd, respectively.  Then $DO(n)=DE(n+1)$. \end{lemma}

  \begin{proof} We define a bijection $\varphi:DO(n)\to DE(n+1)$.  First, $\varphi$ takes the empty partition to the empty partition. Since in both cases the partitions are required to have distinct parts, they are determined by the first column of their abacus.

    Let $\lambda\in DO(n)$ be nonempty;  then its $n$-abacus has at least one run of white beads.  All runs of white beads in its $n$-abacus have an even length, except the last run has an odd length.  We define the $n+1$-abacus of $\varphi(\lambda)$ by its first column -- it is the same as that of $\lambda$, except with one white bead added to its last run.  The result clearly has distinct parts and all runs of white beads having even length, and hence is in $DE(n+1)$.
    Clearly $\varphi$ is invertible, with $\varphi^{-1}$ removing one row from the $n+1$-abacus.
    \end{proof}

  There are two cases.  The simple case is if all the white circles are in the first column of the $n$-abacus of $\lambda$ (that is, $f(k)\leq 1$).  Given such a simultaneous core with all parts odd, we may simply add one white bead to odd run of beads, adding an extra row to the abacus, and getting the abacus diagram of an $(n+1, n+2)$-core with simultaneous part.  Similarly, given a simultaneous core with all parts even, whose abacus is all in the first column, we may just remove a white dot from the last run of white dots and get an $(n, n+1)$ core with all parts odd. 

  Let $SE(n), SO(n)$ be the number of $(n, n+1)$-cores, whose $n$ abacus has all white circles in the first column, and all parts even, or odd, respectively.  The previous paragraph shows $SO(n)=SE(n+1)$.  On the other hand, we can derive a recursion for $SE(n)$ just as in our main theorem, by decomposing over the location of the second black circle in the first column.  This gives
  $$SE(n)=SE(n-1)+SE(n-3)+SE(n-5)+\cdots=SE(n-1)+SE(n-2)$$
  by removing the first term of the sum and using the identity on the remaing terms.  Note though that $SE(0)=SE(1)=SE(2)=1$, as the second equation only holds if $n\geq 3$.  Thus, $SO(n-1)=SE(n)=F(n)$.

  The complicated case is when the abacus diagram of $\lambda$ has white dots on at least two columns (that is, there is some $x$ with $f(x)\geq 2$.  Denote $CE(n), CO(n)$, the number of such $(n, n+1)$-cores with all parts even (or odd, respectively).

  If $\lambda\in CO(n)$, we may not simply insert a single extra row into our abacus diagram with a white dot in the last column and get a partition in $CE(n+1)$.  Although doing so would make the last run of white circles have an even number, the runs of white circles in the previous columns would now have \emph{odd} number of beads.

  In terms of the function $f$, suppose the last run was in the $m$th column, from rows $\ell+1$ to $\ell+k$; i.e., $f(\ell)=m-1, f(\ell+1)=f(\ell+2)=\cdots=f(\ell+k)=m$, and $f(\ell+k+i)<m$ for $i\geq 0$.  Then the description above is attempts to define a new function $g$ by
  $$g(i)=\begin{cases} f(i) & \mbox{if } i\leq \ell+k \\ m & \mbox{if } i=\ell+k+1 \\ f(i-1) & \mbox{if }i\geq \ell+k+2 \end{cases}$$
The final run of white circles in the abacus of $g$ does indeed have one more than that of $f$, but the run of beads in the $m-1$ column just before also has one more white bead in it; since this run was by assumption even before, it is now odd, a problem.

  To fix this, we will insert \emph{two} new rows into the abacus diagram.  One of these rows will add a white bead to the last run of white beads; the other other one will add a black bead immediately or after this run, and have white beads in all previous columns.  Thus, we've added one white bead to the very last run of whites, but to every other run of white beads we've either added no white beads or two white beads, and the result is indeed in $CE(n+2)$.

  However, there are two different ways we could add two columns as above -- we can stick the black bead immediately \emph{before} the last run, or immediately \emph{after} the last run.  In terms of the function $f$ defining the abacus, these two new functions $g_b$ and $g_a$ (for before and after) are
  $$g_b(i)=\begin{cases} f(i) & \mbox{if } i\leq \ell \\
  m-1 & \mbox{if } i=\ell+1 \\
  m & \mbox{if } i=\ell+2 \\
  f(i-2) & \mbox{if }i\geq \ell+3 \end{cases}$$

  $$g_a(i)=\begin{cases} f(i) & \mbox{if } i\leq \ell+k \\
  m & \mbox{if } i=\ell+k+1 \\
  m-1 & \mbox{if } i=\ell+k+2 \\
  f(i-2) & \mbox{if }i\geq \ell+k+3 \end{cases}$$
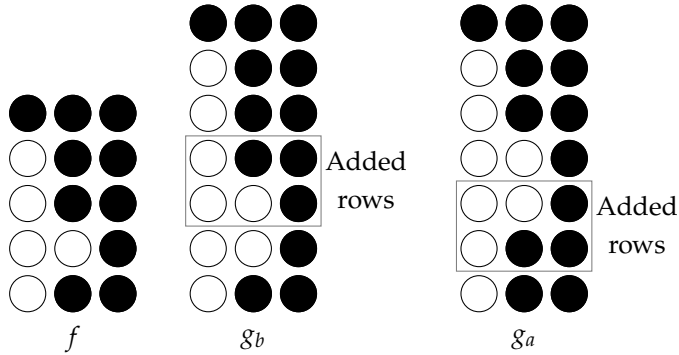
\begin{tikzpicture}[scale=.6]  
\foreach \x in {0,1,...,2}           
\foreach \y in {0,1,...,4}
\draw (\x,\y) circle (.4);
\begin{scope}
  \clip (-.5,4.5)--(-.5,3.5)--(.5,3.5)--(.5,1.5)--(1.5,1.5)--(1.5,.5)--(.5,.5)--(.5,-.5)--(2.5,-.5)--(2.5,4.5)--(-.5,4.5);
\foreach \x in {0,1,...,2}           
\foreach \y in {0,1,...,4}
\filldraw (\x,\y) circle (.4);
\end{scope}
\node at (1,-1) {$f$};

\begin{scope}[xshift=4cm]
\foreach \x in {0,1,...,2}           
\foreach \y in {0,1,...,6}
\draw (\x,\y) circle (.4);
\begin{scope}
  \clip (-.5,6.5)--(-.5,5.5)--(.5,5.5)--(.5,2.5)--(1.5,2.5)--(1.5,.5)--(.5,.5)--(.5,-.5)--(2.5,-.5)--(2.5,6.5)--(-.5,6.5);
\foreach \x in {0,1,...,2}           
\foreach \y in {0,1,...,6}
\filldraw (\x,\y) circle (.4);
\end{scope}
\draw[thin, gray] (-.5,1.5)--(-.5,3.5)--(2.5,3.5)--(2.5,1.5)--cycle;
\node at (1,-1) {$g_b$};
\node at (3.5,2.5) {\begin{tabular}{c}Added \\ rows\end{tabular}};
  
\end{scope}

\begin{scope}[xshift=10cm]
\foreach \x in {0,1,...,2}           
\foreach \y in {0,1,...,6}
\draw (\x,\y) circle (.4);
\begin{scope}
  \clip (-.5,6.5)--(-.5,5.5)--(.5,5.5)--(.5,3.5)--(1.5,3.5)--(1.5,1.5)--(.5,1.5)--(.5,-.5)--(2.5,-.5)--(2.5,6.5)--(-.5,6.5);
\foreach \x in {0,1,...,2}           
\foreach \y in {0,1,...,6}
\filldraw (\x,\y) circle (.4);
\end{scope}
\draw[thin, gray] (-.5,.5)--(-.5,2.5)--(2.5,2.5)--(2.5,.5)--cycle;
\node at (1,-1) {$g_a$};
\node at (3.5,1.5) {\begin{tabular}{c}Added \\ rows\end{tabular}};
\end{scope}

\end{tikzpicture}

  Thus, there are $2CO(n)$ abacus diagrams in $CE(n+2)$ obtained in this way.  However,  some of the abacus diagrams in $CE(n+2)$ are obtained twice in this way, while others aren't obtained at all.

  An abacus diagram in $CE(n+2)$ is obtained twice in this way if both immediately before and following this run there's an ``extra'' black bead in that column. More precisely, if the last run of white circles is in the $m$th column from $\ell+1$ to $\ell+k$, then we must have $f(\ell)=m-1$ and $f(\ell+1)=\cdots=f(\ell+k=m$.  We say that there is an extra black bead \emph{before} if in addition $f(\ell-1)=m-1$, and we say there is an extra black bead \emph{after} if in addition $f(\ell+k+1=m-1)$.

  In case there is an extra black bead both before and after, we may remove both the rows have the abacus having these ``extra'' black beads and get an abacus diagram in $CE(n)$, and vice versa.  Thus, we are double counting $CE(n)$ different elements.

\begin{tikzpicture}[scale=.6]
\foreach \x in {0,1,...,2}           
\foreach \y in {0,1,...,6}
\draw (\x,\y) circle (.4);
\begin{scope}
  \clip (-.5,6.5)--(-.5,5.5)--(.5,5.5)--(.5,3.5)--(1.5,3.5)--(1.5,1.5)--(.5,1.5)--(.5,-.5)--(2.5,-.5)--(2.5,6.5)--(-.5,6.5);
\foreach \x in {0,1,...,2}           
\foreach \y in {0,1,...,6}
\filldraw (\x,\y) circle (.4);
\end{scope}
\node at (1,-1) {\begin{tabular}{c} extra before \\ and after \end{tabular}};
\draw[thin, gray] (-.5,.5)--(2.5,.5)--(2.5,1.5)--(-.5,1.5)--cycle;
\draw[thin, gray] (-.5,4.5)--(-.5,5.5)--(2.5,5.5)--(2.5,4.5)--cycle;

\begin{scope}[xshift=4cm]
  \foreach \x in {0,1,...,2}           
\foreach \y in {0,1,...,4}
\draw (\x,\y) circle (.4);
\begin{scope}
  \clip (-.5,4.5)--(-.5,3.5)--(.5,3.5)--(.5,2.5)--(1.5,2.5)--(1.5,1.5)--(.5,1.5)--(.5,-.5)--(2.5,-.5)--(2.5,4.5)--(-.5,4.5);
\foreach \x in {0,1,...,2}           
\foreach \y in {0,1,...,4}
\filldraw (\x,\y) circle (.4);
\end{scope}
\node at (1,-1) {\begin{tabular}{c}add two \\ before this\end{tabular}};
\end{scope}

\begin{scope}[xshift=8cm]
  \foreach \x in {0,1,...,2}           
\foreach \y in {0,1,...,4}
\draw (\x,\y) circle (.4);
\begin{scope}
  \clip (-.5,4.5)--(-.5,3.5)--(.5,3.5)--(.5,1.5)--(1.5,1.5)--(1.5,.5)--(.5,.5)--(.5,-.5)--(2.5,-.5)--(2.5,4.5)--(-.5,4.5);
\foreach \x in {0,1,...,2}           
\foreach \y in {0,1,...,4}
\filldraw (\x,\y) circle (.4);
\end{scope}
\node at (1,-1) {\begin{tabular}{c}add two \\ after this\end{tabular}};
\end{scope}

\begin{scope}[xshift=12cm]
  \foreach \x in {0,1,...,2}           
\foreach \y in {0,1,...,4}
\draw (\x,\y) circle (.4);
\begin{scope}
  \clip (-.5,4.5)--(-.5,3.5)--(.5,3.5)--(.5,2.5)--(1.5,2.5)--(1.5,.5)--(.5,.5)--(.5,-.5)--(2.5,-.5)--(2.5,4.5)--(-.5,4.5);
\foreach \x in {0,1,...,2}           
\foreach \y in {0,1,...,4}
\filldraw (\x,\y) circle (.4);
\end{scope}
\node at (1,-1) {\begin{tabular}{c}extra rows\\ removed\end{tabular}};
\end{scope}

\end{tikzpicture}  
  The diagrams in $CE(n+2)$ that aren't obtained at all are those that don't have an ``extra'' black bead before OR after the last run of white beads.   In this case, we see that immediately above the black bead that's the upper bound of this run, we must have another run of white beads; otherwise, the second to last column would have an odd run of white beads.

More long-windedly in terms of the function $f$: since the run of white beads is the last one, we must have $f(\ell+k+1\leq m-2)$.  What about $f(\ell-1)$?  Since $f(\ell)=m-1$, we must have $f(\ell-1)\geq m-2$. Since hte last row of beads is at height $m$, we must have $f(\ell-1)\leq m$.  Since there is no extra bead in front of the run by assumption, we cannot have $f(\ell-1)=m-1$.  Finally, we cannot have $f(\ell-1)=m-2$, as then we'd have an odd length run of beads in the $m-1$st column in rows $\ell$ to $\ell+k$.  Therefore, we must have $f(\ell-1)=m$.

\begin{tikzpicture}[scale=.6]
  \foreach \x in {0,1,...,2}           
\foreach \y in {0,1,...,11}
\draw (\x,\y) circle (.4);
\begin{scope}
  \clip (-.5,11.5)--(-.5,10.5)--(.5,10.5)--(.5,9.5)--(1.5,9.5)--(1.5,5.5)--(.5,5.5)--(.5,4.5)--(1.5,4.5)--(1.5,.5)--(-.5,.5)--(-.5,-.5)--(2.5,-.5)--(2.5,11.5)--(-.5,11.5);
\foreach \x in {0,1,...,2}           
\foreach \y in {0,1,...,11}
\filldraw (\x,\y) circle (.4);
\end{scope}
\draw[thin, gray] (-.5,.5)--(-.5,6.5)--(2.5,6.5)--(2.5,.5)--cycle;

\end{tikzpicture}

Thus, in this case we can delete the whole final run of white beads (which has length $2c$ for some even $c$), the black bead bounding the bead above this, and the last white bead from this run before (hence deleting $2c+2$ rows), and get a partition in $CO(n-2c)$.  In terms of the function $f$, we are defining a new function $h$ by

$$h(i)=\begin{cases} f(i) & \mbox{if } i\leq \ell-2 \\
  f(i+2c+2) & \mbox{if } i\geq \ell-1  \end{cases}$$
However, there's the added wrinkle that we don't get all partitions in $CO(n-2k)$ in this way, but only those with no ``extra'' black dots after the last run of white beads -- the function $h$ has its last odd length run ending at $h(\ell-2)=f(\ell-2)=m$, but $h(\ell-1)=f(\ell+2c+1)\leq m-2$.

But if we have an abacus diagram of a partition in $CO(n-2k)$ that \emph{does} have a ``extra'' black dot, after the last run of white beads, we may just turn that extra black dot to white, obtaining an arbitrary diagram in $CE(n-2k)$.
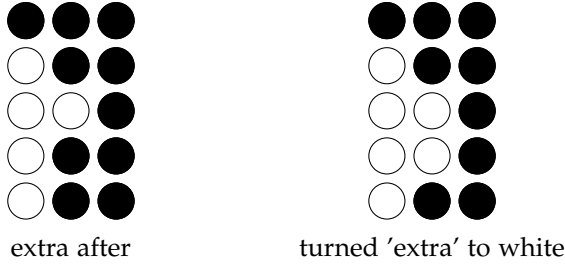
\begin{tikzpicture}[scale=.6]
\foreach \x in {0,1,...,2}           
\foreach \y in {0,1,...,4}
\draw (\x,\y) circle (.4);
\begin{scope}
  \clip (-.5,4.5)--(-.5,3.5)--(.5,3.5)--(.5,2.5)--(1.5,2.5)--(1.5,1.5)--(.5,1.5)--(.5,-.5)--(2.5,-.5)--(2.5,4.5)--(-.5,4.5);
\foreach \x in {0,1,...,2}           
\foreach \y in {0,1,...,4}
\filldraw (\x,\y) circle (.4);
\end{scope}
\node at (1,-1) {extra after};

\begin{scope}[xshift=8cm]
\foreach \x in {0,1,...,2}           
\foreach \y in {0,1,...,4}
\draw (\x,\y) circle (.4);
\begin{scope}
  \clip (-.5,4.5)--(-.5,3.5)--(.5,3.5)--(.5,2.5)--(1.5,2.5)--(1.5,.5)--(.5,.5)--(.5,-.5)--(2.5,-.5)--(2.5,4.5)--(-.5,4.5);
\foreach \x in {0,1,...,2}           
\foreach \y in {0,1,...,4}
\filldraw (\x,\y) circle (.4);
\end{scope}
\node at (1,-1) {turned 'extra' to white};
\end{scope}

\end{tikzpicture}

That is, we've seen the number of partitions in $CE(n+2)$ we've failed to count in $2CO(n)$ is $\sum (CO(n-2k)-CE(n-2k))$.

  Putting it all together, we have:
 $$CE(n+2)=2CO(n)-CE(n)+\sum_{k\geq 1} (CO(n-2k)-CE(n-2k))$$
  or
  $$CO(n)=\sum_{k\geq 0}\big( CE(n+2-2k)-CO(n-2k)\big)$$
  Removing the first term of the sum, reindexing, and reapplying the equation, gives
  \begin{align*}
    CO(n)&=CE(n+2)-CO(n)+\sum_{k\geq 0}\big( CE(n-2+2-2k)-CO(n-2-2k)\big)\\
    &=CE(n+2)-CO(n)-CO(n-2)
    \end{align*}
  and hence $CE(n+2)=2O(n)-CO(n-2)$.

  To get the corresponding identity for $O(n)$ and $E(n)$, we need to add back the one column simultaneous cores counted by $SE(n)$ and $SO(n)$.  But since $SO(n-1)=SE(n)=F(n)$, we have
 $$SE(n+2)=F(n+2)=F(n+1)+F(n)=2F(n+1)-F(n-1)=2SO(n)-SO(n-2)$$

\end{proof}
\section{Enriched recursions} \label{sec:enriched} It is natural to not just study the number of simultaneous cores, but various functions on them.  We briefly summarize some results about how these play with our basic recursion.

The most obvious are the size $|\lambda|$, length (or number of parts) $\ell(\lambda)$, and largest part $\lambda_1$. In particular, we can enrich the generating function $C^M(q)$ to count the partitions $\lambda$ according to any or all of these variables.  
Armstrong, Hanusa and Jones \cite{AHJ} define $(q,t)$-rational Catalan numbers are the sum
Paramonov \cite{Paramonov} studies this sum over partitions with distinct parts.

The three basic statistics can all be read off easily from the abacus diagram:
\begin{itemize}
\item The length $\ell(\lambda)$ is the number of black beads coming before at least one white bead
\item The largest part $\lambda_1$ is the number of white beads on the abacus
  \item The size $|\lambda|$ is the number of inversions, that is, pairs $(a, b)$ where $a$ is a black bead, $b$ is a white bead, and $a$ comes before $b$
  \end{itemize}

When $0\notin M$, and the partition has distinct parts, all of these statistics play well with our basic recurrence.

Let $TL_M(q)$ be the generating function for the sum of all the largest parts, $TP_M(q)$ the total lengths, and $TS_M(q)$ the generating function for the total sizes of all parts, respectively.  Then
$$TL_M(q)=q^2\frac{d}{dq}\chi_M(q)F_M(Q)\frac{1}{1-q\chi_M(q)}$$
$$TP_M(q)=q\chi_M(q)F_M(q)\frac{1}{1-q\chi_M(q)}$$
$$TS_M(q)=(q^2\frac{d}{dq}\chi_M(q)F_M(q)+q\chi_MTL_M(q))\frac{1}{1-q\chi_M(q)}$$
$$TS_M=TL_M\frac{1}{1-q\chi_M}$$

$$TL=\frac{q^2\chi^\prime}{(1-q)(1-q\chi)^2}$$
  $$TP=\frac{q\chi}{(1-q)(1-q\chi)^2}$$
  $$TS=\frac{q^2\chi^\prime}{(1-q)(1-q\chi)^3}$$
 
\section{Beyond $(n, n+1)$-cores} \label{sec:ab}

  \bibliographystyle{plain}
  \bibliography{simultaneouscores}

\end{document}